\documentclass[12pt,a4paper]{article}%
\usepackage{amsmath}
\usepackage{amsfonts}
\usepackage{amssymb}
\usepackage{graphicx}%
\providecommand{\U}[1]{\protect\rule{.1in}{.1in}}
\newtheorem{theorem}{Theorem}

\newtheorem{lemma}[theorem]{Lemma}

\newenvironment{proof}[1][Proof]{\noindent\textbf{#1.} }{$\hfill\Box$}
\begin{document}

\title{A quasilinear problem with fast growing gradient{\thanks{2010 Mathematics
Subject Classification: 35B09, 35J66, 35J70, 35J92}}}
\author{\textbf{{\large Hamilton Bueno and Grey Ercole}}\thanks{The authors were
supported in part by FAPEMIG and CNPq, Brazil.}\\\textit{{\small Departamento de Matem\'{a}tica}}\\\textit{{\small Universidade Federal de Minas Gerais}}\\\textit{{\small Belo Horizonte, Minas Gerais, 30.123.970, Brazil}}\\\textit{{\small e-mail: hamilton@mat.ufmg.br}}, \thinspace
\textit{{\small grey@mat.ufmg.br}}}
\maketitle

\begin{abstract}
In this paper we consider the following Dirichlet problem for the
$p$-Laplacian in the positive parameters $\lambda$ and $\beta$:
\[
\left\{
\begin{array}
[c]{rcll}%
-\Delta_{p}u & = & \lambda h(x,u)+\beta f(x,u,\nabla u) & \text{in }\Omega\\
u & = & 0 & \text{on }\partial\Omega,
\end{array}
\right.  \hfill
\]
where $h,f$ are continuous nonlinearities satisfying $0\leq\omega
_{1}(x)u^{q-1}\leq h(x,u)\leq\omega_{2}(x)u^{q-1}$ with $1<q<p$ and $0\leq
f(x,u,v)\leq\omega_{3}(x)u^{a}|v|^{b}$, with $a,b>0$, and $\Omega$ is a
bounded domain of $\mathbb{R}^{N},$ $N\geq3.$ The functions
$\omega_{i}$, $1\leq i\leq3$, are nonnegative, continuous weights in
$\overline{\Omega}$. We prove that there exists a region $\mathcal{D}$ in the
$\lambda\beta$-plane where the Dirichlet problem has at least one positive
solution. The novelty in this paper is that our result is valid for
nonlinearities with growth higher than $p$ in the gradient variable.\medskip

\end{abstract}


\textit{keywords:} $p$-Laplacian, positive solution, nonlinearity depending on
the gradient, sub- and super-solution method.

\section{Introduction}

Dependence on the gradient in problems involving quasilinear operators as the
$p$-Laplacian have been challenging researchers of elliptic PDE's in questions
of existence and uniqueness. The approach used to handle this problems varies,
ranging from change of variables in order to eliminate the dependence on the
gradient to a combination of topological and blow-up arguments \cite{ubilla,
BEFZ, figubilla, ILS}. In a nutshell, no general method to deal with this kind
of problem has been established.

In this paper we intend to show how simple techniques (sub- and super-solution
method combined with a global estimate on the gradient) are able to solve some
quasilinear problems involving nonlinearities with fast growing gradient, that
is, nonlinearities where the exponent of $\left\vert \nabla u\right\vert $ is
greater than $p$. This type of problem is rare in the literature. The method
we choose allows us to make simple hypotheses, also in contrast with papers in
the area.

We consider the following Dirichlet problem in the positive parameters
$\lambda$ and $\beta:$
\[
\left\{
\begin{array}
[c]{rcll}%
-\Delta_{p}u & = & \lambda h(x,u)+\beta f(x,u,\nabla u) & \text{in }\Omega\\
u & = & 0 & \text{on }\partial\Omega.
\end{array}
\right.  \hfill
\]

Our hypotheses on $f$ (see below) include nonlinearities that depend on the
gradient with an exponent higher than $p$ and thus the application of
variational methods (even in combination with topological techniques, see
\cite{djairo}) can not handle directly this kind of problem. Known versions of
the sub- and super-solution method developed for equations depending on the
gradient (see \cite{Boccardo, mabel}) require nonlinearities with the gradient
term growing at most as $\left\vert \nabla u\right\vert ^{p}.$

Here, inspired by the classical paper of Ambrosetti, Brezis and Cerami
\cite{ABC} we define a fixed point operator for each $\left(  \lambda
,\beta\right)  $ in a region $\mathcal{D}$ of the $\lambda\beta$-plane and use
global $C^{1,\alpha}$ estimates on the solution of the Poisson
equation$-\Delta_{p}u=g$ with Dirichlet boundary conditions on $\Omega$ to
obtain an invariant subset by this operator. \ Hence, by applying Schauder's
fixed point theorem we prove the existence of at least one positive solution
for the Dirichlet problem above if $\left(  \lambda,\beta\right)
\in\mathcal{D}.$


\section{Existence of a positive solution}

In this section we consider the existence of positive solutions for the
following problem in two positive parameters in the bounded, smooth domain
$\Omega\subset\mathbb{R}^{N}:$%

\begin{equation}
\left\{
\begin{array}
[c]{rcll}%
-\Delta_{p}u & = & \lambda h(x,u)+\beta f(x,u,\nabla u) & \text{in }\Omega\\
u & = & 0 & \text{on }\partial\Omega,
\end{array}
\right.  \label{twotwo}%
\end{equation}
where $\Delta_{p}u:=\operatorname{div}(\left\vert \nabla u\right\vert
^{p-2}\nabla u)$ is the $p$-Laplacian operator, $p>1,$ and $h,f$ are
continuous nonlinearities satisfying

\begin{enumerate}
\item[(H1)] $0\leq\omega_{1}(x)u^{q-1}\leq h(x,u)\leq\omega_{2}(x)u^{q-1}$,
$1<q<p$;

\item[(H2)] $0\leq f(x,u,v)\leq\omega_{3}(x)u^{a}|v|^{b}$, $a,b>0,$
\end{enumerate}

\noindent and $\omega_{i}\colon\overline{\Omega}\rightarrow\lbrack0,\infty)$,
$1\leq i\leq3$, are nonnegative continuous functions (with $\omega
_{i}\not \equiv 0$) that we call \emph{weights}.

We begin establishing a version of a result on the regularity of solutions of
the $p$-Laplacian, which was proved by Tolksdorf \cite{Tolks} and Liebermann
\cite{LIEBERMAN}. The proof is given, since the result is not explicitly stated
in those papers.
\begin{lemma}
\label{Lem1}\textit{Let }$\Omega$\textit{ be a bounded, smooth domain of
}$\mathbb{R}^{N}$\textit{ and }$g\in L^{\infty}(\Omega).$\textit{ Assume that
}$u\in W_{0}^{1,p}(\Omega)$\textit{ is a weak solution of}
\begin{equation}
\left\{
\begin{array}
[c]{rcll}%
-\Delta_{p}u & = & g & \text{in }\ \Omega,\\
u & = & 0 & \text{on }\ \partial\Omega.
\end{array}
\right.  \label{Dirig}%
\end{equation}

\end{lemma}

\textit{Then there exists a positive constant }$\mathcal{K}$\textit{,
depending only on }$p,N$\textit{ and }$\Omega,$\textit{ such that}%
\begin{equation}
\left\Vert \nabla u\right\Vert _{\infty}\leq\mathcal{K(}\left\Vert
g\right\Vert _{\infty})^{\frac{1}{p-1}}. \label{estimate}%
\end{equation}

\begin{proof}
Let us firstly assume that $\left\Vert g\right\Vert _{\infty}=1.$ By applying
a simple comparison principle, one can easily verify that $\left\vert
u\right\vert \leq\phi$ where $\phi\in W_{0}^{1,p}(\Omega)\cap L^{\infty
}(\Omega)$ is the $p$-torsion function of $\Omega,$ that is, $-\Delta_{p}%
\phi=1$ in $\Omega.$ Therefore,
\begin{equation}
\left\Vert u\right\Vert _{\infty}\leq L:=\left\Vert \phi\right\Vert _{\infty}.
\label{boundL}%
\end{equation}

It follows from global regularity results by Lieberman (see \cite{LIEBERMAN})
that there exist constants $\alpha\in(0,1)$ and $\mathcal{K}>0$ such that
$u\in C^{1,\alpha}(\overline{\Omega})$ and $\left\Vert u\right\Vert
_{1,\alpha}\leq\mathcal{K}$ and, moreover, $\alpha$ and $\mathcal{K}$ depend
only on $p,N$ and $\Omega.$ (In principle, these constansts could also depend
on the bound $L$ for $\left\Vert u\right\Vert _{\infty},$ but as we easily
see, the bound in (\ref{boundL}) is uniform with respect to $u$ whenever
$\left\Vert g\right\Vert _{\infty}=1$).

Since $\left\Vert \nabla u\right\Vert _{\infty}\leq\left\Vert u\right\Vert
_{1,\alpha}$ we obtain (\ref{estimate}) in the case $\left\Vert g\right\Vert
_{\infty}=1.$ If $0<\left\Vert g\right\Vert _{\infty}\neq1$ we apply the
previous argument to the function $v:=\dfrac{u}{(\left\Vert g\right\Vert
_{\infty})^{\frac{1}{p-1}}}$ since this function satisfies%
\begin{equation}
\left\{
\begin{array}
[c]{rcll}%
-\Delta_{p}v & = & g/\left\Vert g\right\Vert _{\infty} & \text{in }\ \Omega,\\
v & = & 0 & \text{on }\ \partial\Omega.
\end{array}
\right.  \label{v}%
\end{equation}
Thus, we obtain%
\[
\frac{\left\Vert \nabla u\right\Vert _{\infty}}{(\left\Vert g\right\Vert
_{\infty})^{\frac{1}{p-1}}}\leq\mathcal{K}.
\]

Therefore, we have proved (\ref{estimate}) for any $0\not \equiv g\in
L^{\infty}(\Omega)$ where the positive constant $\mathcal{K}$ depends only on
$p,$ $N$ and $\Omega.$ Obviously, (\ref{estimate}) remains valid for the same
constant $\mathcal{K}$ if $g\equiv0.$
\end{proof}\vspace{.5cm}

To solve problem (\ref{twotwo}) we define
\[
r:=a+b+1,\quad\omega(x):=\max_{i\in\left\{  1,2,3\right\}  }\omega_{i}(x)
\]
and denote by $\lambda_{1}$ and $u_{1}$ the first eigenpair of the
$p$-Laplacian with weight $\omega_{1}$, that is,
\[
\left\{
\begin{array}
[c]{rcll}%
-\Delta_{p}u_{1} & = & \lambda_{1}\omega_{1}u_{1}^{p-1} & \text{in \ }%
\Omega,\\
u_{1} & = & 0 & \text{on \ }\partial\Omega,
\end{array}
\right.
\]
with $u_{1}$ positive satisfying $\left\Vert u_{1}\right\Vert _{\infty}=1$.

Let also $\phi\in W_{0}^{1,p}(\Omega)\cap C^{1,\alpha}(\overline{\Omega})$ be
the solution of the problem
\[
\left\{
\begin{array}
[c]{rcll}%
-\Delta_{p}\phi & = & \omega & \text{in \ }\Omega\\
\phi & = & 0 & \text{on \ }\partial\Omega
\end{array}
\right.
\]
and define
\[
\gamma:=\frac{\mathcal{K}\left\Vert \omega\right\Vert _{\infty}^{\frac{1}%
{p-1}}}{\left\Vert \phi\right\Vert _{\infty}},
\]
where $\mathcal{K}$ satisfies (\ref{estimate}). We stress that $\gamma$
depends only on $\omega$, $p$, $N$ and $\Omega.$

\begin{lemma}
\label{Lem2}There exists a region $\mathcal{D}$ in the $\lambda\beta$-plane
such that, if $(\lambda,\beta)\in\mathcal{D}$ then
\begin{equation}
\lambda M^{q-1}+\beta\gamma^{b}M^{a+b}\leq(M/\left\Vert \phi\right\Vert
_{\infty})^{p-1}, \label{M}%
\end{equation}
for some positive constant $M.$
\end{lemma}

\begin{proof}
The inequality (\ref{M}) can be written as
\begin{equation}
\Phi(M):=\lambda AM^{q-p}+\beta BM^{r-p}\leq1, \label{h(M)}%
\end{equation}
where the coefficients
\begin{equation}
A=\left\Vert \phi\right\Vert _{\infty}^{p-1}\quad\text{and}\quad
B:=\mathcal{K}^{b}\left\Vert \phi\right\Vert _{\infty}^{p-1-b}\left\Vert
\omega\right\Vert _{\infty}^{\frac{b}{p-1}} \label{A,B}%
\end{equation}
clearly depend only on $\omega,$ $p$ and $\Omega.$

In order to determine an adequate value for $M$, we consider the possibilities
for the sign of $r-p.$

\noindent\textbf{Case 1:} $r-p>0.$ In this case we have
\[
\lim_{t\rightarrow0^{+}}\Phi(t)=\lim_{t\rightarrow+\infty}\Phi(t)=+\infty
\]
implying that $\Phi$ has a minimum value.
Since the only critical point $M$ of $\Phi$ is given by
\begin{equation}
M:=\left[  \frac{\lambda A(p-q)}{\beta B(r-p)}\right]  ^{\frac{1}{r-q}},
\label{M*case1}%
\end{equation}
we obtain
\[
\Phi(M)=\frac{\beta B(r-p)M^{r-p}}{p-q}+\beta BM^{r-p}=\beta BM^{r-p}\left(
\frac{r-q}{p-q}\right)  \leq\Phi(t)\text{ \ for all }t\geq0.
\]

Now we need to find sufficient conditions on $\lambda$ and $\beta$ in order to
obtain $\Phi(M)\leq1$ or, equivalently,%
\[
\beta B\left[  \frac{\lambda A(p-q)}{\beta B(r-p)}\right]  ^{\frac{r-p}{r-q}%
}\left(  \frac{r-q}{p-q}\right)  \leq1.
\]
After rewriting this last inequality we arrive at%
\begin{equation}
\lambda^{r-p}\beta^{p-q}\leq\left(  \frac{r-p}{A}\right)  ^{r-p}\left(
\frac{p-q}{B}\right)  ^{p-q}\frac{1}{(r-q)^{r-q}}=:K. \label{lbcase1}%
\end{equation}

Thus, if the positive parameters $\lambda$ and $\beta$ satisfy (\ref{lbcase1}%
), we conclude that $\overline{u}:=(M/\left\Vert \phi\right\Vert _{\infty
})\phi$ is a super-solution for (\ref{Tu}), where $M$ is given by
(\ref{M*case1}).

\noindent\textbf{Case 2:} $r-p=0.$ In this case $\Phi(t):=\lambda
At^{q-p}+\beta B$ is positive, strictly decreasing and satisfies%
\[
\lim_{t\rightarrow0^{+}}\Phi(t)=+\infty\quad\text{and}\quad\lim_{t\rightarrow
+\infty}\Phi(t)=\beta B.
\]
So, in order to have $\Phi(M)\leq1$ for some $M>0$ it is necessary that $\beta
B<1.$ Thus,
\begin{equation}
\text{if }\lambda>0\ \text{and}\ \beta<B^{-1} \label{lbcase2}%
\end{equation}
we can take $M>0$ such that $\Phi(M)=1,$ that is
\begin{equation}
M=\left(  \frac{\lambda A}{1-\beta B}\right)  ^{\frac{1}{p-q}}.
\label{M*case2}%
\end{equation}
Thus, if $\lambda$ and $\beta$ satisfy (\ref{lbcase2}) then $\overline
{u}=(M/\left\Vert \phi\right\Vert _{\infty})\phi$ where $M$ is given by
(\ref{M*case2}).

\noindent\textbf{Case 3:} $r-p<0.$ It follows from (\ref{h(M)}) that $\Phi$ is
strictly decreasing and
\[
\lim_{t\rightarrow0^{+}}\Phi(t)=+\infty\text{ \ and \ }\lim_{t\rightarrow
+\infty}\Phi(t)=0.
\]
Hence, for any positive parameters $\lambda$ and $\beta,$ there always exists
$M>0$ such that
\[
\Phi(M)=\lambda AM^{q-p}+\beta BM^{r-p}=1
\]
and for such a $M$ the function $\overline{u}=(M/\left\Vert \phi\right\Vert
_{\infty})\phi$ is a super-solution of (\ref{Tu}).

Summarizing, we have proved that there exists a positive constant $M$
satisfying (\ref{M}) whenever the pair $(\lambda,\beta)\ $belongs to the set
$\mathcal{D}$ defined by:%
\begin{equation}
\mathcal{D}=\left\{
\begin{array}
[c]{lll}%
\left\{  \lambda,\beta>0:\lambda^{r-p}\beta^{p-q}\leq K\right\}  & \text{if} &
r-p>0,\\
\left\{  \lambda,\beta>0:\beta<B^{-1}\right\}  & \text{if} & r-p=0,\\
\left\{  \lambda,\beta>0\right\}  & \text{if} & r-p<0,
\end{array}
\right.  \label{setD}%
\end{equation}

where $K$ and $B$ were determined by (\ref{lbcase1}) and (\ref{A,B}), respectively.
\end{proof}\vspace{.5cm}

For each $u\in C^{1}(\overline{\Omega})$ we define the continuous nonlinearity
$F^{u}:\overline{\Omega}\times\mathbb{R}\rightarrow\mathbb{R}$ by
\begin{equation}
F^{u}(x,\xi):=\lambda\omega_{1}\xi^{q-1}+\lambda\left(  h(x,u(x))-\omega
_{1}u(x)^{q-1}\right)  +\beta f(x,u(x),\nabla u(x)). \label{F^u}%
\end{equation}
and observe that $F^{u}(x,u)=\lambda h(x,u)+\beta f(x,u,\nabla u).$

Our main result of existence of solution for problem (\ref{twotwo}) is given by

\begin{theorem}
\label{maint} Assume that $h$ and $f$ are continuous and satisfy \textup{(H1)}
and \textup{(H2)}. There exists a region $\mathcal{D}$ in the $\lambda\beta
$-plane such that if $\left(  \lambda,\beta\right)  \in\mathcal{D}$ the
Dirichlet problem (\textup{\ref{twotwo}}) has at least one positive solution
$u$ satisfying, for some positive constants $\epsilon$ and $M$:
\[
\epsilon u_{1}\leq u\leq(M/\left\Vert \phi\right\Vert _{\infty})\phi\text{
\ \ and \ \ }\left\Vert \nabla u\right\Vert _{\infty}\leq\gamma M.
\]

\end{theorem}

\begin{proof}
Let $(\lambda,\beta)\in\mathcal{D}$ where the region $\mathcal{D}$ is defined
by (\ref{setD}) and take $M>0$ satisfying (\ref{M}) from Lemma \ref{Lem2}. Let
us define the subset%
\begin{equation}
\mathcal{F}:=\left\{  u\in C^{1}(\overline{\Omega}):\epsilon u_{1}\leq
u\leq(M/\left\Vert \phi\right\Vert _{\infty})\phi\text{ \ and \ }\left\Vert
\nabla u\right\Vert _{\infty}\leq\gamma M\right\}  \subset C^{1}%
(\overline{\Omega}) \label{setF}%
\end{equation}
where%
\[
0<\epsilon\leq\min\left\{  \left(  \lambda/\lambda_{1}\right)  ^{\frac{1}%
{p-q}},(M\lambda_{1}^{-\frac{1}{p-1}})/\left\Vert \phi\right\Vert _{\infty
}\right\}  .
\]

We divide this proof into five steps.

\noindent\textbf{Step 1. } We prove that for each $u\in\mathcal{F}$
there
exists a positive solution $U$ of the problem
\begin{equation}
\left\{
\begin{array}
[c]{rcll}%
-\Delta_{p}U & = & F^{u}(x,U) & \text{in \ }\Omega\\
U & = & 0 & \text{on \ }\partial\Omega
\end{array}
\right.  \label{Tu}%
\end{equation}
satisfying
\[
\epsilon u_{1}\leq u\leq(M/\left\Vert \phi\right\Vert _{\infty})\phi.
\]
In order to do this we firstly verify that the functions
\[
\underline{u}:=\epsilon u_{1}\text{ \ \ \ and \ \ }\overline{u}:=(M/\left\Vert
\phi\right\Vert _{\infty})\phi
\]
constitute an ordered pair of sub- and super-solutions of (\ref{Tu}). This fact
implies, by applying a standard iteration process, that there exists a weak
solution $U$ of (\ref{Tu}) satisfying $\underline{u}\leq U\leq\overline{u}.$

Since $\overline{u}$ satisfies
\begin{equation}
\left\{
\begin{array}
[c]{rcll}%
-\Delta_{p}\overline{u} & = & \omega(M/\left\Vert \phi\right\Vert _{\infty
})^{p-1}, & \text{in \ }\Omega\\
\overline{u} & = & 0, & \text{on \ }\partial\Omega
\end{array}
\right.  \label{plapsuper}%
\end{equation}
$u\in\mathcal{F}$ and $\left\Vert \overline{u}\right\Vert _{\infty}=M,$ we
obtain from (\ref{M}) of Lemma \ref{Lem2} that
\begin{align*}
F^{u}(x,\overline{u})  &  \leq\lambda\omega_{1}\overline{u}^{q-1}%
+\lambda\left(  \omega_{2}-\omega_{1}\right)  u^{q-1}+\beta\omega_{3}%
u^{a}\left\vert \nabla u\right\vert ^{b}\\
&  \leq\lambda\omega_{1}M^{q-1}+\lambda\left(  \omega_{2}-\omega_{1}\right)
M^{q-1}+\beta\omega_{3}M^{a}\gamma^{b}M^{b}\\
&  \leq\omega\left(  \lambda M^{q-1}+\beta\gamma^{b}M^{a+b}\right) \\
&  \leq\omega(M/\left\Vert \phi\right\Vert _{\infty})^{p-1}=-\Delta
_{p}\overline{u}.
\end{align*}
Hence, the weak comparison principle gives that $\overline{u}$ is a
super-solution of (\ref{Tu}).

Now, since $\underline{u}=0$ on $\partial\Omega$ and
\[
-\Delta_{p}\underline{u}=\lambda_{1}\omega_{1}\underline{u}^{p-1}\leq
\lambda_{1}\omega_{1}\epsilon^{p-q}\underline{u}^{q-1}\leq\lambda\omega
_{1}\underline{u}^{q-1}\leq F^{u}(x,\underline{u})\text{ \ in }\Omega,
\]
we obtain from the weak comparison principle again that $\underline{u}$ is a
sub-solution for (\ref{Tu}). This principle still produces the ordering
$\underline{u}\leq\overline{u}$ in $\overline{\Omega}$, since
\[
-\Delta_{p}\underline{u}\leq\epsilon^{p-1}\lambda_{1}\omega_{1}\leq
\epsilon^{p-1}\lambda_{1}\omega\leq\left(  M_{\ast}/\left\Vert \phi\right\Vert
_{\infty}\right)  ^{p-1}\omega=-\Delta_{p}\overline{u}.
\]

\noindent\textbf{Step 2.} Now we complete the verification that
$U\in\mathcal{F}$ by proving that $\left\vert \nabla U\right\vert \leq\gamma
M.$ Indeed, it follows from (\ref{estimate}) of Lemma \ref{Lem1} that%
\[
\left\Vert \nabla U\right\Vert _{\infty}^{p-1}\leq\mathcal{K}^{p-1}\left\Vert
F^{u}(x,U)\right\Vert _{\infty}%
\]
and from (H1), (H2) and (\ref{M}) that
\[%
\begin{array}
[c]{lll}%
0\leq & F^{u}(x,U) & =\lambda\omega_{1}U^{q-1}+\lambda\left(  h(x,u)-\omega
_{1}u^{q-1}\right)  +\beta f(x,u,\nabla u)\\
&  & \leq\lambda\omega_{1}U^{q-1}+\lambda\left(  \omega_{2}-\omega_{1}\right)
u^{q-1}+\beta\omega_{3}u^{a}\left\vert \nabla u\right\vert ^{b}\\
&  & \leq\lambda\omega_{2}^{q-1}(M\phi/\left\Vert \phi\right\Vert _{\infty
})^{q-1}+\beta\omega_{3}(M\phi/\left\Vert \phi\right\Vert _{\infty}%
)^{a}(\gamma M)^{b}\\
&  & \leq\omega(\lambda M^{q-1}+\beta\gamma^{b}M^{a+b})\\
&  & \leq\left\Vert \omega\right\Vert _{\infty}(M/\left\Vert \phi\right\Vert
_{\infty})^{p-1}=(\gamma M/\mathcal{K})^{p-1}.
\end{array}
\]

\noindent\textbf{Step 3.} We prove the uniqueness of $U$. It is a
consequence of a result proved in \cite{DiazSaa}, but it also follows from
Picone's inequality (see \cite{Allegretto})
\[
\left\vert \nabla u\right\vert ^{p}\geq\left\vert \nabla v\right\vert
^{p-2}\nabla v\cdot\nabla\left(  \frac{u}{v}\right),
\]
which is valid for all differentiable $u\geq0$ and $v>0.$ In fact, if $U$ and
$V$ are both \emph{positive} solutions of problem (\ref{Tu}), we have
\[
\int_{\Omega}\frac{F^{u}(x,U)U^{p}}{U^{p-1}}dx=\int_{\Omega}|\nabla
U|^{p}dx\geq\int_{\Omega}|\nabla V|^{p-2}\nabla V\cdot\nabla\left(
\frac{U^{p}}{V^{p-1}}\right)  dx=\int_{\Omega}\frac{F^{u}(x,V)U^{p}}{V^{p-1}%
}dx,
\]
from what follows
\[
\int_{\Omega}\left(  \frac{F^{u}(x,U)}{U^{p-1}}-\frac{F^{u}(x,V)}{V^{p-1}%
}\right)  U^{p}dx\geq0.
\]
An analogous inequality is also true for $V$:
\[
-\int_{\Omega}\left(  \frac{F^{u}(x,U)}{U^{p-1}}-\frac{F^{u}(x,V)}{V^{p-1}%
}\right)  V^{p}dx\geq0,
\]
and so
\begin{equation}
\int_{\Omega}\left(  \frac{F^{u}(x,U)}{U^{p-1}}-\frac{F^{u}(x,V)}{V^{p-1}%
}\right)  (U^{p}-V^{p})dx\geq0. \label{+}%
\end{equation}
Since $q<p$, it follows from (\ref{F^u}) that $F^{u}(x,\xi)/\xi^{p-1}$ is
decreasing with respect to $\xi.$ Therefore, the last integrand is
non-positive and so (\ref{+}) yields
\[
\left(  \frac{F^{u}(x,U)}{U^{p-1}}-\frac{F^{u}(x,V)}{V^{p-1}}\right)
(U^{p}-V^{p})=0\text{ \ in }\Omega
\]
from what we obtain $U=V.$

\noindent\textbf{Step 4.} The regularity $U\in C^{1,\alpha}(\overline{\Omega
})$ for some $0<\alpha<1$ uniform with respect to $u\in\mathcal{F}$ follows from the
uniform boundedness of both $U$ and $\left\vert \nabla U\right\vert $ together
with classical results (see \cite{DiBenedetto, LIEBERMAN, Tolks}). We
emphasize that the bounds for $U$ and $\left\vert \nabla U\right\vert $ are
determined by the positive constant $M$ which, in its turn, is fixed according
with the pair $(\lambda,\beta)\in\mathcal{D}.$

\noindent\textbf{Step 5.} In this last step we complete the proof. As
consequence of the previous steps the operator
\[%
\begin{array}
[c]{rcl}%
T\colon\mathcal{F}\subset C^{1}(\overline{\Omega}) & \longrightarrow &
C^{1,\alpha}(\overline{\Omega})\cap W_{0}^{1,p}(\Omega)\subset C^{1}%
(\overline{\Omega})\\
u & \longrightarrow & U,
\end{array}
\]
is well-defined, $U$ being the unique positive solution of (\ref{Tu}).
Moreover, it follows clearly from the compactness of the immersion
$C^{1,\alpha}(\overline{\Omega})\hookrightarrow C^{1}(\overline{\Omega})$ that
$T$ is continuous and compact. Thus, since $T$ leaves invariant the set
$\mathcal{F}$ defined by (\ref{setF}) and this set is bounded and convex we
can apply Schauder's Fixed Point Theorem to obtain a fixed point $u$ for $T.$
Of course, such a fixed point $u$ satisfies%
\[
\left\{
\begin{array}
[c]{rcll}%
-\Delta_{p}u & = & F^{u}(x,u)=\lambda h(x,u)+\beta f(x,u,\nabla u) &
\text{in\ }\Omega\\
u & = & 0 & \text{on \ }\partial\Omega.
\end{array}
\right.
\]

\end{proof}

\noindent \textbf{Acknowledgement} The authors thank B. Sirakov for useful
conversations.


\begin{thebibliography}{99}
                                         %
\bibitem {Allegretto}\textsc{W. Allegretto} and \textsc{Y.X. Huang}, \textit{A
Picone's identity for the p-Laplacian and applications}, \textit{Nonlinear
Analysis} \textbf{32} (1998), 819-830.

\bibitem {ABC}\textsc{A. Ambrosetti, H. Brezis} and \textsc{G. Cerami},
\textit{Combined effects of concave and convex nonlinearities in some elliptic
problems}, J. Functional Analysis \textbf{122} (1994), 519-543.

\bibitem {Boccardo}\textsc{L. Boccardo, F. Murat} and \textsc{J.-P. Puel},
\textit{R\'{e}sultats d'existence pour certains probl\`{e}mes elliptiques
quasilin\'{e}aires}, Ann. Scuola Norm. Sup. Pisa Cl. Sci. (4), \textbf{11}
(1984), no 2, 213-235.

\bibitem {ubilla}\textsc{F. Brock, L. Iturriaga} and \textsc{P. Ubilla},
\textit{Semi-linear singular elliptic equations with dependence on the
gradient}, Nonlinear Anal. \textbf{65} (2006), no. 3, 601-614.

\bibitem {BEFZ}\textsc{H. Bueno, G. Ercole, W.M. Ferreira} and \textsc{A.
Zumpano}, \textit{Existence of positive solutions for the $p$-Laplacian with
dependence on the gradient}, Nonlinearity \textbf{25} (2012), 1211-1234.




\bibitem {mabel}\textsc{M. Cuesta Leon}, \textit{Existence results for
quasilinear problems via ordered sub and super-solutions}, Ann. Fac. Sci.
Toulouse \textbf{6} (1997), no 4, 591-608.

\bibitem {djairo}\textsc{D. de Figueiredo, M. Girardi} and \textsc{M. Matzeu},
\textit{Semilinear elliptic equations with dependence on the gradient via
mountain-pass techniques}, Differential Integral Equations \textbf{17 }(2004),
no. 4, 119-126.

\bibitem {figubilla}\textsc{D. de Figueiredo, J. S\'{a}nchez} and \textsc{P.
Ubilla}, \textit{Quasilinear equations with dependence on the gradient},
Nonlinear Anal. \textbf{71} (2009), no. 10, 4862-4868.

\bibitem {DiazSaa}\textsc{J. I. D\'{\i}az} and \textsc{J. E. Saa},
\textit{Existence et unicit\'{e} de solutions positives pour certaines
\'{e}quations elliptiques quasilin\'{e}aires} C. R. Acad. Sci. Paris
\textbf{305} S\'{e}rie I (1987) 521--524.

\bibitem {DiBenedetto}\textsc{E. DiBenedetto}, $C^{1,\alpha}$\textit{ local
regularity of weak solutions of degenerate elliptic equations}, Nonlinear
Analysis \textbf{7} (1998), 827-850.











\bibitem {ILS}\textsc{L. Iturriaga, S. Lorca} and \textsc{J. S\'{a}nchez},
\textit{Existence and multiplicity results for the $p$-Laplacian with a
$p$-gradient term}, NoDEA Nonlinear Diff. Equations Appl. \textbf{15} (2008),
no. 6, 729-743.


\bibitem {LIEBERMAN}\textsc{G.M. Lieberman}, \textit{Boundary regularity for
solutions of degenerate elliptic equations}, Nonlinear Anal. \textbf{12}
(1988), no. 1, 1203-1219.



\bibitem {Tolks}\textsc{P. Tolksdorf}, \textit{Regularity for a more general
class of quasilinear elliptic equations}, J. Differential Equations
\textbf{51} (1984), 126-150.
\end{thebibliography}
\end{document}